\documentclass[12pt]{amsart}
\usepackage{amsmath,amssymb,amsthm,color, euscript, enumerate,comment}
\usepackage{dsfont}
\usepackage{longtable}

\renewcommand{\l}{\left}
\renewcommand{\r}{\right}

\newcommand{\maru}[1]{{\ooalign{\hfil#1\/\hfil\crcr
\raise.167ex\hbox{\mathhexbox20D}}}}

\newcommand{\ruby}[2]{%
 \leavevmode
 \setbox0=\hbox{#1}%
 \setbox1=\hbox{\tiny #2}%
 \ifdim\wd0>\wd1 \dimen0=\wd0 \end{lemma}se \dimen0=\wd1 \fi
 \hbox{%
   \kanjiskip=0pt plus 2fil
   \xkanjiskip=0pt plus 2fil
   \vbox{%
     \hbox to \dimen0{%
       \tiny \hfil#2\hfil}%
     \nointerlineskip
     \hbox to \dimen0{\mathstrut\hfil#1\hfil}}}}

\newcommand{\la}{\langle}
\newcommand{\ra}{\rangle}

\newcommand{\Z}{\mathbb{Z}}
\newcommand{\C}{\mathbb{C}}

\newcommand{\Q}{\mathbb{Q}}

\newcommand{\aut}{\mathrm{Aut}\,}
\newcommand{\Aut}{\mathrm{Aut}\,}

\renewcommand{\hom}{\mathrm{Hom}}

\newcommand{\be}{\beta}
\newcommand{\al}{\alpha}

\newcommand{\Span}{\mathrm{Span}}

\makeatletter \@addtoreset{equation}{section}

\theoremstyle{plain}
\newtheorem{theorem}{Theorem}[section]
\newtheorem{proposition}[theorem]{Proposition}
\newtheorem{lemma}[theorem]{Lemma}

\theoremstyle{definition}
\newtheorem{definition}[theorem]{Definition}

\theoremstyle{remark}
\newtheorem{remark}[theorem]{Remark}
\numberwithin{equation}{section}

\usepackage[symbol]{footmisc}

\title[Fourvolutions and orbifold of lattice VOAs]{Fourvolutions  and automorphism groups of orbifold  lattice vertex operator algebras}

\author{Hsian-Yang Chen$^{*}$}
  \address[H.Y.  Chen]{ National University of Tainan, Tainan  70005, Taiwan}
\email{hychen@mail.nutn.edu.tw} 
 \thanks{$*$Corresponding author}
 \thanks{H.Y. Chen is supproted by MOST grant 109-2115-M-024 -005 -MY2 of Taiwan}
\author{Ching Hung Lam} %
\address[C. H. Lam] {Institute of Mathematics, Academia Sinica, Taipei 10617, Taiwan} 
\email{chlam@math.sinica.edu.tw} 
\date{}
\thanks{C.H. Lam is supported by a research grant AS-IA-107-M02 of Academia Sinica  and MOST grant  107-2115-M-001 -003 -MY3  of Taiwan}

\newcommand{\sfr}[2]{\leavevmode\kern-.1em
  \raise.5ex\hbox{\the\scriptfont0 #1}\kern-.1em
  /\kern-.15em\lower.25ex\hbox{\the\scriptfont0 #2}}

\begin{document}

\begin{abstract}
Let $L$ be an even positive definite lattice with no roots, i.e., $L(2)=\{x\in L\mid (x|x)=2\}=\emptyset$. Let 
$g\in O(L)$ be an isometry of order $4$ such that $g^2=-1$ on $L$. In this article, we determine the full automorphism group of the orbifold vertex operator algebra $V_L^{\hat{g}}$.  As our main result, we show that 
$\aut(V_L^{\hat{g}})$ is isomorphic to  $N_{\Aut(V_L)}(\langle \hat{g}\rangle)/ \langle\hat{g}\rangle $ unless $L\cong \sqrt{2}E_8$ or $BW_{16}$.  
\end{abstract}

\maketitle


\section{Introduction}
Let $V$ be a vertex operator algebra (abbreviated  as VOA) and let $g \in \Aut V$  be  a finite automorphism of $V$. The fixed-point subalgebra $$V^g = \{ v \in V\mid  g v = v \}$$ is often 
called an orbifold subVOA. Our main propose is to study the full automorphism group of the orbifold VOA $V^g$. It is clear that the normalizer $N_{\aut(V)}(\langle g\rangle)$ of  $g$ stabilizes  the orbifold VOA $V^g$ and it induces a group homomorphism $f:N_{\aut(V)}(\langle g\rangle)/ \langle g\rangle \to \aut(V^g)$. For generic cases, $\aut(V^g)$ is often isomorphic to $N_{\aut(V)}(\langle g\rangle)/ \langle g\rangle$ but $\aut(V^g)$ may be strictly bigger than  $N_{\aut(V)}(\langle g\rangle)/ \langle g\rangle$. We call an automorphism $s\in \aut(V^g)$ an extra automorphism if $s$ is not in the image of $f$, or equivalently, $s$  is not a restriction from an automorphism in $\Aut(V)$.   
When $V=V_L$ is a lattice VOA and $g$ is a lift of the $(-1)$ isometry of $L$, the full automorphism group of $V_L^g =V_L^+$ has been determined in \cite{Sh04,Sh06}.  In particular, it was shown in \cite{Sh04,Sh06} that the VOA $V_L^+$ contains an extra automorphism if and only if $L$ can be constructed by Construction B from a binary code (cf. Section \ref{s:4.1}). The main method is to study the orbit
of $V_{L}^- =\{v\in V_L\mid gv=-v\}$ under the conjugate actions of $\Aut(V_{L}^+)$. Using a similar method, the automorphism groups of certain orbifold vertex operator algebras  associated with some coinvariant lattice of the Leech lattice are also studied in \cite{Lam1,Lam2}. 
In addition, a sufficient condition for the existence of extra automorphisms was discussed in \cite{Lam1,LS}. The key idea is  to generalize  the triality automorphism defined in \cite{FLM} (see also Section \ref{s:4.1} and \eqref{extra_auto}) to other root lattices of type $A$.   
These  results were also used in \cite{BLS,BLS2} to determine the automorphism groups of certain orbifold vertex operator algebras  associated with the Leech lattice and to determine the full automorphism groups of certain holomorphic VOAs of central charge $24$.  
One interesting question is if the sufficent condition mentioned in \cite{LS} is also a necessary condition for the existence of extra automorphisms?  
When $|g|=p$ is a prime and $L$ is rootless, it was shown in \cite{LS} that a cyclic orbifold $V_L^{\hat{g}}$ contains extra automorphisms if and only if the rootless even lattice $L$ can be constructed by Construction B from a code over $\Z_p$ or is isometric to the coinvariant lattice of the Leech lattice associated with a certain isometry of order $p$. 

In this article, we will continue our study of the full automrohism groups of certain cyclic orbifold of lattice vertex operator algebras and consider a case that $g\in O(L)$ has order $4$ and $g^{2}=-1$ on $L$. Such an isometry is called a fourvolution in \cite{Gr2tod}.  As our main result, we will show that 
$\aut(V_L^{\hat{g}})$ is isomorphic to  $N_{\Aut(V_L)}(\langle \hat{g} \rangle)/ \langle \hat{g}\rangle$ for any fourvolution $g\in O(L)$ unless $L\cong \sqrt{2}E_8$ or $BW_{16}$. When $L\cong \sqrt{2}E_8$ or $BW_{16}$, we also have 
$\aut(V_L^{\hat{g}})\cong  C_{\Aut(V_L^+)}(\bar{g})/ \langle \bar{g}\rangle$,    
where $\bar{g}$ denotes the restriction of $\hat{g}$ on $V_L^+$.

We shall note that for the cases studied in \cite{BLS,BLS2,Lam1,Lam2,LS}, $g$ often acts as an $n$-cycle on some root lattices of type $A_{n-1}$ and thus all $n$-th roots of unity (except $1$) appear as  eigenvalues of $g$, where $n=|g|$. Nevertheless, for a fourvolution $g\in O(L)$, $g$ has only two eigenvalues $\sqrt{-1}$ and $-\sqrt{-1}$ and $-1$ is not an eigenvalue of $g$. Therefore, the condition discussed in \cite{LS} would never be satisfied when $g$ is a fourvolution. We hope that our study of fourvolutions and the automorphism groups of the corresponding orbifold VOA will provide some hints for the general cases.

\section{Preliminaries}
\subsection{Lattice VOAs and their automorphism groups}
We first recall the structure of the automorphism group of a lattice VOA $V_L$.   
Let $L$ be an even lattice with the (positive-definite) bilinear form $( \ | \ )$.
We denote the group of isometries of $L$ by $O(L)$, i.e, 
\[
O(L)=\{ g\in GL(L)\mid (gx|gy)=(x|y) \text{ for all } x,y\in L\}.
\]
Let $\hat{L}=\{\pm e^\al\mid \al\in L\}$ be  a 
central extension of $L$ by $\pm 1$
such that $e^{\al} e^\be = (-1)^{(\al |\be)} e^{\be}e^\al$.
Let $\aut(\hat{L})$ be the automorphism group of $\hat{L}$ as a group.
For $g\in\aut(\hat{L})$, let $\bar{g}$ be the map $L\to L$ defined by $g(e^\alpha)\in\{\pm e^{\bar{g}(\alpha)}\}$.
Let $O(\hat{L})=\{g\in\aut(\hat{L})\mid \bar{g}\in O(L)\}$.
Then by \cite[Proposition 5.4.1]{FLM},  we have an exact sequence
\[
  1 \to \hom (L,\Z/2\Z) \to  O(\hat{L})\to  O(L) \to  1.
\]
It is known that $O(\hat{L})$ is a subgroup of $\aut(V_L)$ (cf.~loc.~cit.).
Let
$$
  N(V_L) = \l\la \exp(a_{(0)}) \mid a\in (V_L)_1 \r\ra
$$
be the normal subgroup of $\Aut(V_L)$ generated by the inner automorphisms $\exp(a_{(0)})$.

\begin{theorem}[\cite{DN}]\label{aut}
Let $L$ be a positive definite even lattice.
Then
\[
  \aut (V_L) = N(V_L)\,O(\hat{L})
\]
Moreover, the intersection $N(V_L)\cap O(\hat{L})$ contains a subgroup
$\hom (L,\Z/2\Z)$ and the quotient $\aut (V_L)/N(V_L)$ is isomorphic
to a quotient group of $O(L)$.
\end{theorem}

\begin{remark}\label{L2=0}
  If $L(2)=\emptyset$, then $(V_L)_1=\Span\{ \alpha{(-1)} \cdot 1 \mid \alpha\in L\}$.
  In this case, the normal subgroup
  $N(V_L)= \{\exp(\lambda \alpha{(0)}) \mid \alpha\in L,~\lambda \in \C\}$ is abelian and
  we have $N(V_L) \cap O(\hat{L})= \hom (L,\Z/2\Z)$. Moreover, $\aut (V_L)/N(V_L) \cong O(L)$ and we have an exact sequence
  \begin{equation}\label{eq:5.3}
    1\to N(V_L) \to \aut(V_L) \stackrel{~\varphi~}\to O(L)\to 1.
  \end{equation}
  Note also that $\exp(\lambda \al_{(0)})$ acts trivially on $M(1)$ and
  $\exp(\lambda \alpha{(0)}) e^\beta = \exp(\lambda \langle \alpha, \beta\rangle) e^{\beta}$ for 
any
  $\lambda \in \C$ and $\alpha, \beta \in L$.
\end{remark}

The following theorem can be proved by the same argument as in \cite[Theorem 
5.15]{LY2}.
 
\begin{theorem}
 \label{normalizer}
Let $L$ be a positive-definite rootless even lattice.
Let $g$ be a fixed-point free isometry of $L$ of finite order and
$\hat{g}$ a lift of $g$ in $O(\hat{L})$.
Then we have the following exact sequences:
\[
\begin{split}
 1\longrightarrow \hom(L/(1-g)L, \C^*)
  \longrightarrow N_{\aut(V_L)}(\langle \hat{g}\rangle)
  \stackrel{\varphi}\longrightarrow N_{O(L)}(\langle g\rangle) \longrightarrow 1;\\
   1\longrightarrow \hom(L/(1-g)L, \C^*)
  \longrightarrow C_{\aut(V_L)}(\hat{g})
  \stackrel{\varphi}\longrightarrow C_{O(L)}(g) \longrightarrow 1. 
\end{split}
\]
\end{theorem}

Let $h\in N_{\aut(V_L)}(\langle \hat{g}\rangle )$. Then it is clear that $hx \in V_L^{\hat{g}}$ for any $x\in 
V_L^{\hat{g}}$. Therefore, $N_{\aut(V_L)}(\langle \hat{g}\rangle)$ acts on $V_L^{\hat{g}}$ and there is a group 
homomorphism $$f: N_{\aut(V_L)}(\langle \hat{g}\rangle)/ \langle \hat{g}\rangle \longrightarrow 
\aut(V_L^{\hat{g}}).$$ The key question is to determine the image of $f$ and if $f$ is surjective. In 
particular, one would like to determine if there exist automorphisms in $\aut(V_L^{\hat{g}})$ 
which are not induced from $N_{\aut(V_L)}(\langle \hat{g}\rangle)$. We call such an 
automorphism an \emph{extra automorphism}.

\begin{definition}
An element $g\in O(L)$ is called a fourvolution  if $g^2=-1$ on $L$. 
\end{definition}

\begin{lemma}[\cite{Gr2tod}]
If $f$ is a fourvolution of a lattice $L$, then the adjoint of $1\pm f$ is $1\mp f$ and $1\pm f$ is an isometry scaled by $\sqrt{2}$. Moreover, 
we have $L\geq (1-f)L \geq 2L$ and $|L : (1-f)L| = |(1-f)L : 2L| = |L/2L|^{1/2}$. In particular, $\mathrm{rank}(L)$ is even.
\end{lemma}

The main purpose of this article is to determine the full automorphism group of $V_L^{\hat{g}}$ when $g$ is a fourvolution and $L(2)=\{x\in L\mid (x|x)=2\} = \emptyset$. In particular, we will show that $V_L^{\hat{g}}$ contains no extra automorphisms unless $L\cong \sqrt{2}E_8$ or $BW_{16}$.  

\section{Orbifolds of lattice VOAs having extra automorphisms }
Let $L$ be an even lattice with $L(2)=\emptyset$. 
Let $g\in O(L)$ be a fourvolution, i.e., $g^2=-1$. 
In this section, we will study $\Aut(V_L^{\hat{g}})$. In particular, we will determine when  $\Aut(V_L^{\hat{g}})$ contains an extra automorphism.

First we recall that the orbifold VOA $V_{L}^{\hat{g}}$ is  $C_2$-cofinite and rational \cite{Mi,CM} and  any irreducible $V_{L}^{\hat{g}}$-module is a submodule of an irreducible $\hat{g}^i$-twisted $V_{L}$-module for some $0\le i\le |\hat{g}|-1$  \cite{DRX}.

For any irreducible (untwisted or twisted) module $M$ of $V_L$, $M$ is $\hat{g}$-stable if 
$M\circ \hat{g}\cong M$. In this case, $\hat{g}$ acts on $M$ and we use 
\[
M(j)= \{ x\in M\mid  \hat{g} x= e^{2\pi \sqrt{-1}j / n} x\}, \quad  0 \leq j\leq n-1,\ n=|\hat{g}|,   
\]
to denote the eigenspaces of $\hat{g} $ on $M$.  Notice that $V_{\lambda+L}(j)$ is a simple current module of $V_L^{\hat{g}}$ if $V_{\lambda+L}$ is $\hat{g}$-stable, or equivalently, $(1-g)\lambda\in L$ (see for example \cite{Lam19}). By a result in \cite{DLM2}, it is known that the number of inequivalent  irreducible $\hat{g}^i$-twisted modules is equal to the number of inequivalent irreducible $\hat{g}^i$-stable modules of $V_L$. Moreover, all irreducible $\hat{g}^i$-twisted modules are $\hat{g}$-stable. 
 
Let $(L^*/L)^{g^i}$ be the set of cosets of $L$ in $L^*$ fixed by $g^i$ and let $P_0^{g^i}:L^* \to \Q\otimes_\Z L^{g^i}$ be the orthogonal projection. 
Then  $V_L$ has exactly $|(L^*/L)^{g^i}|$ irreducible $\hat{g}^i$-twisted $V_L$-modules, up to isomorphism. The irreducible $\hat{g}^i$-twisted $V_L$-modules have been constructed in \cite{Le,DL} explicitly and are classified in \cite{BK04}.
They are given by 
\begin{equation}
V_{\lambda+L}[\hat{g}^i]= M(1)[g^i]\otimes\C[P_0^{g^i}(\lambda+L)]\otimes T_{\tilde{\lambda}}, 
\qquad \text{ for } \lambda+L\in (L^*/L)^{g^i},   \label{twmodule0}
\end{equation}
where $M(1)[g^i]$ is the ``$g^i$-twisted" free bosonic space, $\C[\lambda+P_0^{g^i}(L)]$ is a module for the group algebra of $P_0^{g^i}(L)$ and $T_{\tilde{\lambda}}$ is an irreducible module for a certain ``$g^i$-twisted" central extension of $L_{g^i}$ associated with $\lambda$ (see \cite[Propositions 6.1 and 6.2]{Le} and \cite[Remark 4.2]{DL} for detail).

Now assume that $V_L^{\hat{g}}$ has an extra automorphism $\sigma$.  
Then by \cite[Theorem 2.1]{Sh07} and $V_L(0)\circ\sigma=V_L^{\hat{g}}\circ\sigma\cong V_L^{\hat{g}}=V_L(0)$, we have
\[
\{ V_L(r)\circ \sigma \mid 1\leq r\leq 3\} \neq \{ V_L(r) \mid 1\leq r\leq 3\}.  
\]
In other words, $V_L(1)\circ \sigma$ is isomorphic to a simple current module of $V_L^{\hat{g}}$ 
not containing in $\{ V_L(r) \mid 1\leq r\leq 3\}$. 
By the classification of simple current irreducible $V_L^{\hat{g}}$-modules,   $V_L(1)\circ \sigma$ is either isomorphic to 
\begin{enumerate}[(I)]
\item  $V_{\lambda+L}(r)$ for some $\lambda\in L^*\setminus L$ with $(1-g)\lambda\in L$ and $0\leq r\leq 3$; 
 
\item an irreducible $V_L^{\hat{g}}$-submodule $V_{\lambda+L}^T[\hat{g}^2](j)$ for some $ 0\leq j \leq 3$; or 

 \item an irreducible $V_L^{\hat{g}}$-submodule $V_{\lambda+L}^T[\hat{g}^s]_\Z$ with integral weights for some $s=1$ or $3$.
\end{enumerate}


\subsection{Case (I): $\sigma$-conjugation of $V_L(1)$ is of untwisted type}\label{S:1}

By the assumption, $V_{\lambda+L}(r)$ is a simple current module of $V_L^{\hat{g}}$ and hence $(1-g)\lambda \in L$ \cite[Theorem 4.11]{Lam19}. Since  $g^2=-1$  and there are only two eigenvalues ($\sqrt{-1}$ and $-\sqrt{-1}$)  of $g$ on $\C\otimes_\Z L$, we have  
$$\dim V_L(1)_1= \dim V_L(3)_1 = \frac{\mathrm{rank}\,L}{2}.$$
Hence, we have $\dim V_{\lambda+L}(r)_1= \mathrm{rank}\,L/{2}$, also. 
Moreover,  we have $\dim V_{\lambda+L}(r)_1=|(\lambda+L)(2)|/4$ for any $0\leq r\leq 3$ and thus  
\begin{equation}
|(\lambda+L)(2)|= 2\cdot\mathrm{rank}\,L.\label{Eq:lambda2}
\end{equation}

Now set $N={\rm Span}_\Z\{L, \lambda\}$.
\begin{lemma}
We have $|N/L|=2$. 
\end{lemma}

\begin{proof}
First we note that $(1-g)\mu \in L$ for any $\mu \in \lambda+L$. In particular, 
$g\lambda \in \lambda+L$ and $(1-g)g\lambda= g\lambda -g^2\lambda=g\lambda+\lambda\in L\cap (2\lambda+L)$. It implies $2\lambda\in L$. 
\end{proof}

\begin{lemma}\label{Lem:A1n}
The sublattice of $N$ spanned by $N(2)$ is isometric to the orthogonal sum of $n$ copies of $A_{1}$, where $n=\mathrm{rank}\,L$.   
\end{lemma}

\begin{proof}
Let $\alpha \in (\lambda+L)(2)$. Then 
\[
(\alpha|g\alpha) =(g\alpha|g^2\alpha)= -(g\alpha|\alpha);
\]
hence $(\alpha|g\alpha)=0$. Moreover, for any $\beta \in (\lambda+L)(2)$ and $\beta\notin \{\pm \alpha, \pm g\alpha\}$, we have $(\alpha |\beta)= (g\alpha |\beta)=0$; otherwise, $(\alpha|\beta)=\pm 1$. It implies $\alpha+\beta$ or $\alpha-\beta$ are roots in $L$, which is impossible.  Since $|(\lambda+L)(2)|= 2n$, $\lambda+L$ forms a semisimple root system of $A_1^n$ as desired.   
\end{proof}

\begin{remark}
By \cite[Proposition 1.8]{Sh04}, $L$ can be obtained by construction B from a binary code associated with an orthogonal basis in $(\lambda+L)(2)$. 
\end{remark}

Let $\theta=\hat{g}^2$. Then $\theta$ is a lift of the $-1$ isometry of $L$ and we use $V_L^+$ to denote the fixed point VOA $V_L^\theta$. 

\begin{proposition}\label{lifttoVL+}
Let $\sigma$ be an extra automorphism of $V_L^{\hat{g}}$ such that $V_L(1)\circ \sigma \cong V_{\lambda+L}(r)$ for some $0\le r\le 3$. Then $\sigma$ lifts to an automorphism of $V_L^+$.  
\end{proposition}

\begin{proof}
Since $2\lambda\in L$ and $V_{\lambda+L}(r)$ is a simple current module, $V_{\lambda+L}(r)^{\boxtimes 2}$ is isomorphic to an irreducible $V_L^{\hat{g}}$-submodule of $ V_L$.  Moreover, $V_L(1)^{\boxtimes 2} \cong V_L(2)$ has top weight $> 1$. 

Since $V_L(2)$ is the only irreducible  $V_L^{\hat{g}}$-submodule of $ V_L$ which 
has the top weight $>1$, we have $V_{\lambda+L}(r)^{\boxtimes 2}= (V_L(1)\circ \sigma)^{\boxtimes 2} \cong V_L(2)$. Therefore, $\sigma$ preserves the subspace
$V_L^+= V_L^{\hat{g}}+ V_L(2)$ and $\sigma$ lifts to an automorphism of $V_L^+$ by 
 \cite[Theorem 2.1]{Sh07}. 
\end{proof}

\subsection{Case (II): $\sigma$-conjugation of $V_L(1)$ is contained in $\hat{g}^2$-twisted module}\label{S:3}

Next, we consider Case (II), i.e, $V_L(1)\circ \sigma\cong V_{\lambda+L}^T[\hat{g}^2](j) $ for some $0\leq j\leq 3$.

In this case,  $V_{\lambda+L}^T[\hat{g}^2]_1 \neq 0$.   Since $\hat{g}^2$ is a lift of the $-1$ isometry, it is proved in \cite{Sh04} that 
$n=\mathrm{rank}\, L= 8$ or $16$.  Moreover, $L^*/L$ is an elementary abelian $2$-group and $L\cong \sqrt{2} E_8$ if $n=8$ and  $L\cong BW_{16}$ if $n=16$.

\subsection{Case (III): $\sigma$-conjugation of $V_L(1)$ is contained in $\hat{g}$-twisted $V_L$-module}\label{S:2}

Next, we consider Case (III), i.e, $V_L(1)\circ \sigma\cong V_{\lambda+L}^T[\hat{g}^s]_\Z$ for some $s=1$ or $3$. 

Since $g$ is fixed point free on $L$, the irreducible $\hat{g}^s$-twisted module $V_{\lambda+L}^T[\hat{g}^s]$, for $s=1$ or $3$, is given by 
\begin{equation*}
V_{\lambda+L}^T[\hat{g}^s]= M(1)[{g}^s]\otimes T_{\tilde{\lambda}}. 
\end{equation*}

For a fourvolution $g$ and $s$ odd, the conformal weight of $V_{\lambda+L}^T[\hat{g}^s]$ (see \cite{Le,DL}) is given by 
\begin{equation}
\varepsilon=\frac{3n}{4\times 4^2}.\label{Eq:esp}
\end{equation}
That $V_L(1)\circ \sigma\cong V_{\lambda+L}^T[\hat{g}^s]_\Z$  implies 
$\varepsilon \leq 1$ and $\varepsilon \in \frac{1}4 \Z$. It is easy to verify that 
$n=16$ is the only solution and $\varepsilon=3/4$. 
In this case, $V_{\lambda+L}^T[\hat{g}^s]_\Z$ is a simple current module for $V_L^{\hat{g}}$; hence we have $(1-g)L^*\leq L$ and $\dim T_{\tilde{\lambda}} =[L: (1-g)L^*]^{1/2}$ \cite[Corollary 3.7]{ALY}.  
Since  $\varepsilon=3/4$ and $\dim (V_{\lambda+L}^T[\hat{g}^s])_1= \dim V_L(1)_1= {\mathrm{rank}\,L}/{2}$,  we have $\dim T_{\tilde{\lambda}}=1$ and $L=(1-g)L^*$. Moreover, 
\[
|L^*/L|= |L/(1-g)L|=|\det(1-g)| =2^{n/2}.  
\]
For $n=16$, we have $|L^*/L|=2^8$. Note that $L_2=\emptyset$ and there is only one such lattice up to isometry and $L\cong BW_{16}$ \cite[Proposition 1.9]{Sh04}. 

By our assumption, $V_L(1)\circ \sigma\cong V_{\lambda+L}^T[\hat{g}^s]_\Z$ for some $\lambda\in L^*$ and $s=1$ or $3$ and we have 
\[
(V_{\lambda+L}^T[\hat{g}^s]_\Z )^{\boxtimes 2} \cong V_L(2)\circ \sigma. 
\]
On the other hand, $(V_{\lambda+L}^T[\hat{g}^s]_\Z )^{\boxtimes 2}$ is isomorphic to an irreducible  $V_{L}^{\hat{g}}$-submodule of a $\hat{g}^2$-twisted module $V_{\mu+L}^T[\hat{g}^2]$ for some $\mu \in L^*$. 
  
When $L\cong BW_{16}$, the conformal weight of $V_{\mu+L}^T[\hat{g}^2]$  is $1$ for any $\mu\in L^*$ and $V_{\mu+L}^T[\hat{g}^2](j)$, $0 \leq j\leq 3,$ has conformal weight either $1$ or $3/2$.  It  is a contradiction since  the conformal weight of  $V_L(2)$ is $\geq 2$.  That means Case (III) does not occur.

\section{Automorphism groups of $V_L^{\hat{g}}$}
In this section, we will study the automorphism groups of $V_L^{\hat{g}}$ when $g$ is a fourvolution, i.e, $g^2=-1$. Let $\theta=\hat{g}^2$. Then $\theta$ is a lift of the $-1$ isometry of $L$ and we use $V_L^+$ to denote the fixed point VOA $V_L^\theta$.

\subsection{$ L\ncong \sqrt{2}E_8$ nor $BW_{16}$}  \label{s:4.1}
First we assume that $ L\ncong \sqrt{2}E_8$ nor $BW_{16}$. 

\begin{theorem}\label{nbw16}
Let $L$ be even lattice with $L(2)=\emptyset$ and $g\in O(L)$ a fourvolution. 
Suppose $L\ncong\sqrt{2}E_8$ nor $BW_{16}$. Then $\Aut(V_L^{\hat{g}})\cong C_{\Aut(V_L^+)}(\bar{g})/ \langle  \bar{g}\rangle$, where $\bar{g}$ denotes the restriction of $\hat{g}$ on $V_L^+\cong V_L^{\hat{g}^2}$. 
\end{theorem}

\begin{proof}
Suppose $L\ncong\sqrt{2}E_8$ nor $BW_{16}$. Then either $V_L^{\hat{g}}$ has no extra automorphisms or Case (I) holds. In either cases, automorphisms of $V_L^{\hat{g}}$ will stabilize $V_L^{\hat{g}^2}=V_L(0) \oplus V_L(2)$ by Proposition \ref{lifttoVL+}.  Since $V_L^{\hat{g}^2}$ is a simple current extension of $V_L^{\hat{g}}$, the result then follows from  
\cite[Theorem 3.3]{Sh04}. 
\end{proof}

Recall that $g^2=-1$ and hence $\hat{g}^2$ is a lift the $(-1)$-isometry. 
The automorphism group for the fixed point VOA $V_L^+$ has been determined in \cite{Sh04}. We will recall some results in \cite{Sh04}. First, we recall the Construction B of a lattice from a binary code $C$. 

Let $C$ be a doubly even binary code of length $n$ and let $\mathcal{B}=\{\alpha_i \mid i\in\{1, \dots, n\}\} $ be an orthogonal basis of $\mathbb{R}^n$  of norm $2$, i.e, $\langle \alpha_i, \alpha_j\rangle =2\delta_{i,j}$. For $c=(c_1,\dots, c_n)\in \Z_2^n$,  set  $$\alpha_c=  \sum_{i=1}^n c_i\alpha_i.$$   The lattice 
\[
L_B(C) = \sum_{c\in C} \Z \frac{1}2\alpha_c + \sum_{i, j\in \{1, \dots, n\}} \Z (\alpha_i +\alpha_j)
\]
is often referred as to the lattice obtained by Construction B from $C $ associated
with $\mathcal{B}$. 
Note that $L_B(C)$ always contains the  sublattice  $\sum_{i, j\in \{1, \dots, n\}} \Z (\alpha_i +\alpha_j)\cong \sqrt{2}D_n$. 

Set $L_A(C) = L_B(C)+ \mathbb{Z}\alpha_1$. Then $\mathcal{B}=\{\alpha_i| i=1, \dots,n\} < L_A(C)$. Now fix  $a_k\in \widehat{L_A}(C)$ for each $k\in \{1, \dots, n\}$  such that $\bar{a}_k= \alpha_k$ and define 
\begin{equation}\label{extra_auto}
\sigma=\prod_{k=1}^n\exp((1+\sqrt{-2})(a_k)_0)\exp(\sqrt{-\frac{1}{2}}(a_k^{-1})_0)
\exp((-1+\sqrt{-2})(a_k)_0).
\end{equation}
Then $\sigma$ is an automorphism of the VOA $V_{L_A(C)}$. Indeed, $\sigma$ defines an automorphism of $V_{L_B(C)}^+$ (cf. \cite{FLM,Sh04}). 

\begin{remark} We shall note that there are usually several choices for the orthogonal basis $\mathcal{B}$ and 
the automorphism $\sigma$ depends on the choice of the orthogonal basis 
$\mathcal{B}=\{ \alpha_1, \dots, \alpha_n\}$.  Moreover, for each $\alpha_k\in \mathcal{B}$,  
\[
\begin{split}
\sigma :\quad  \alpha_k (-1) &\longmapsto e^{\alpha_k } + e^{-\alpha_k}, \\
e^{\alpha_k } + e^{-\alpha_k} &\longmapsto
 \alpha_k(-1),\\ 
  e^{\alpha_k } - e^{-\alpha_k}&\longmapsto -(e^{\alpha_k } - e^{-\alpha_k}).
\end{split}
\]
\end{remark}
 
The following theorem can be found in \cite{Sh04}.
 
\begin{theorem}[{\cite[Proposition 3.16]{Sh04}}]
Let $L$ be an even lattice such that $L(2)=\emptyset$. 
Then $\Aut(V_L^+)$ is generated by $O(\hat{L})/\langle \theta\rangle$ and the extra automorphisms defined as in \eqref{extra_auto}. In particular, $\Aut(V_L^+)$ contains an extra  automorphism if and only if $L$ can be constructed by  Construction B from some binary code $C$.    
\end{theorem}

\begin{theorem}\label{thm:4.4}
Suppose $L(2)=\emptyset$ and $L\ncong \sqrt{2}E_8$ nor $BW_{16}$. Then for a fourvolution $g\in O(L)$, we have 
\[
\Aut(V_L^{\hat{g}}) \cong N_{\Aut(V_L)}(\langle \hat{g}\rangle)/ \langle\hat{g}\rangle.
\]
\end{theorem}

\begin{proof}
Suppose false. Then there is $\sigma\in \Aut(V_L^{\hat{g}})$ such that 
\[
\{ V_L(i)\circ \sigma \mid 0\leq i\leq 3\} \neq \{ V_L(i) \mid 0\leq i\leq 3\}.
\]
Since $L\ncong \sqrt{2}E_8$ nor $BW_{16}$, only Case (I) can occur; hence, there exist $\lambda\in L^*$ and $0\leq j \leq 3$ such that  $V_L(1) \circ \sigma \cong V_{\lambda+L}(j)$. 

Recall from the discussion in Case (I), we have  $|(\lambda+L)(2)|=2n$ and $(\lambda+L)(2)$ forms a root system of type $A_1^n$. Moreover, $\sigma$ lifts to an automorphism of $V_L^+= V_L(0)\oplus V_L(2)$.  In this case,  $V_L^-=V_L(1) \oplus V_L(3)$ and $V_L^-\circ \sigma\cong V_{\lambda+L}^\varepsilon$ for some $\varepsilon=\pm 1$.  Moreover, $N=\mathrm{Span}_Z\{
L, \lambda\}= L_A(C)$ for some binary code $C$ and $L= L_B(C)$. 

Fix an orthogonal basis $\mathcal{B}_{\lambda} < (\lambda+L)(2)$, i.e., a set of simple roots for $(\lambda+L)(2)$, and let $\sigma_\lambda$ be the extra automorphism associated with the basis $\mathcal{B}_{\lambda}$ as defined above.  Then $\sigma \sigma_{\lambda}^{-1}$  stabilize  $V_L^{-}$ and hence  $\sigma \sigma_{\lambda}^{-1}\in O(\hat{L})$ \cite{Sh04}
and $\sigma = h \sigma_{\lambda}$ for some $h\in O(\hat{L})$. 

Now let $\alpha_1, \alpha_2\in \mathcal{B}_{\lambda}$ such that  $g$ maps $\alpha_1 \to \alpha_2\to -\alpha_1\to -\alpha_2$.  Then $\hat{g}$ fixes 
the element $(e^{\alpha_1} + e^{-\alpha_1})_{-1} (e^{\alpha_2} + e^{-\alpha_2})$. Note that 
$(e^{\alpha_1} + e^{-\alpha_1}), (e^{\alpha_2} + e^{-\alpha_2})\in V_{\lambda+L}$ and hence 
$(e^{\alpha_1} + e^{-\alpha_1})_{-1} (e^{\alpha_2} + e^{-\alpha_2})\in V_{2\lambda+L}=V_L$. However, for any $h\in O(\hat{L})$, 
\[
\begin{split}
&\ h\sigma_{\lambda} ((e^{\alpha_1} + e^{-\alpha_1})_{-1} (e^{\alpha_2} + e^{-\alpha_2})\\
  = &\ h( \alpha_1(-1)\alpha_2(-1) \cdot 1)\\ 
= &\ \bar{h}\alpha_1(-1)\bar{h}\alpha_2(-1)\cdot 1,
\end{split}
\]
which is not fixed by $\hat{g}$.  It contradicts that $h \sigma_{\lambda} =\sigma\in \Aut(V_L^{\hat{g}}) $. 
\end{proof}

\subsection{Case: $L\cong \sqrt{2}E_8$ or $BW_{16}$}
Next we consider the case $L=\sqrt{2}E_8$ or $BW_{16}$. In both cases,  
$L^*/L$ is an elementary abelian $2$-group of order $2^8$. It is also known \cite{Sh04} that the set of inequivalent irreducible modules $\mathrm{Irr}(V_L^+)$ 
of $V_L^+$ forms an elementary abelian $2$-group under the fusion rules \cite{Sh04}. 
Recall that $V_{\lambda+L}^T[\theta]_\Z$ has the conformal weight $1$ if  $L\cong \sqrt{2}E_8$ or $BW_{16}$. Moreover, $\dim(V_{\lambda+L}^T[\theta]_1)=8$ if $L\cong \sqrt{2}E_8$ and 
$\dim(V_{\lambda+L}^T[\theta]_1)=16$ if $L\cong BW_{16}$. 

\begin{lemma}
Let $L=\sqrt{2}E_8$ or $BW_{16}$. 
Suppose there is a $\sigma\in V_{L}^{\hat{g}}$ such that $V_L(1)\circ \sigma\cong V_{\lambda+L}^T[\hat{g}^2](j)$.  
Then $\sigma$ stabilizes $V_L^+= V_L(0) \oplus V_L(2) $ by the conjugate action. 
\end{lemma}

\begin{proof}
For $L\cong \sqrt{2}E_8$ or $BW_{16}$,  we have $V_{\lambda+L}^T[\hat{g}^2]_\Z \boxtimes_{V_L^+} V_{\lambda+L}^T[\hat{g}^2]_\Z = V_L^+$ and hence $V_{\lambda+L}^T[\hat{g}^2](j)^{\boxtimes 2} < V_{L}^+$. Therefore,  $V_{\lambda+L}^T[\hat{g}^2](j)^{\boxtimes 2}\cong V_L(2)$ as $V_L(2)$ is the only irreducible $V_L^{\hat{g}}$ module with the top weight $>1$ in $V_L^+$.
\end{proof}

As a consequence, we have 
\begin{proposition}\label{e8bw16}
Let $L\cong \sqrt{2}E_8$ or $BW_{16}$ and $g\in O(L)$ a fourvolution. 
Then, $$\Aut(V_L^{\hat{g}})\cong C_{\Aut(V_L^+)}(\bar{g})/\langle \bar{g} \rangle,$$ 
where $\bar{g}$ denotes the restriction of $\hat{g}$ on $V_L^+\cong V_L^{\hat{g}^2}$. 
\end{proposition}

\subsubsection{$L\cong \sqrt{2}E_8$}
In this case, the isometry group 
$O(L)$ is the Weyl group of $E_8$ and it has the shape $2.\Omega^+_8(2).2$. 
A fourvolution $g$  corresponds to a $2C$-element in $\Omega^+_8(2)$. 
An irreducible $\hat{g}^2$-twisted module $V_{\lambda+L}^T[\hat{g}^2]$ has the top weight $1/2$ and
$\dim (V_{\lambda+L}^T[\hat{g}^2])_1=8$. 

Recall that the set of inequivalent irreducible modules for $V_L^+$ forms an elementray abelian group of order $2^{10}$ under the fusion rules and the automorphism group $V_{\sqrt{2}E_8}^+ \cong GO^+_{10}(2)$ (cf. \cite{Sh04,GO+102}). 

\begin{theorem}
Let $L\cong \sqrt{2}E_8$ and $g\in O(\sqrt{2}E_8)$ a fourvolution. Then 
$\Aut(V_L^{\hat{g}})$ has order $2^{18}.3^2.5$ and has the shape $[2^{14}]. Sym_6$. 
\end{theorem}

\begin{proof}
We first recall that $\dim(V_{\sqrt{2}E_8}^+)_2=156$ and $(V_{\sqrt{2}E_8}^+)_2$ decomposed as 
$1+155$ as a sum of irreducible representations of $GO^+_{10}(2)$. By a direct calculation, it is easy to  show that 
$\dim(V_{\sqrt{2}E_8}^{\hat{g}})_2 =76$.  Therefore, $\bar{g}$ acts on $(V_{\sqrt{2}E_8}^+)_2 $ with trace $-4$, i.e., $\bar{g}$ corresponds to a $2C$ element of $GO^+_{10}(2)$ using the Atlas notation. In this case, the centralizer $C_{GO^+_{10}(2)}(\bar{g})$ has order $2^{19}.3^2.5$  and has the shape $[2^{15}]. Sym_6$. 
By Proposition \ref{e8bw16}, we have the desired result.
\end{proof}

\subsubsection{$L\cong BW_{16}$}  In this case,  $L^*/L \cong 2^8$ and $|L/2L^*|=2^8$.
  
First we recall few facts about the isometry group $O(L)$. It is clear that $O(L)$ acts naturally on the discriminant group $L^*/L\cong 2^8$ and it induces a group homomorphism 
$s: O(L)\to O(L^*/L, q)\cong GO^+_8(2)$.

Let $K=\ker s$. Then $K$ is isomorphic to an extra special group $2^{1+8}$ and $\mathrm{Im}\,s\cong \Omega^+_8(2)$. Therefore, $O(L)\cong 2^{1+8}.\Omega^+_8(2)$ (see  \cite{CS, Gr2tod}). A fourvolution $g$ can be decomposed into the product of two involutions $t_M, t_N$ associated with two $\sqrt{2}E_8$ sublattices $M$ and $N$ such that $L=M+N$.  Recall that $t_M$ is an isometry of $L$ which acts as $-1$ on $M$ and acts as $1$ on the annihilator $\mathrm{Ann}_L(M)=\{ x\in L\mid (x,y)=0 \text{ for all } y\in M\}$. It is well-defined since $2M^*< M$ (cf. \cite{Gr2tod,GL}).  These two involutions 
$t_M, t_N$ generate a dihedral group of order $8$  and is contained in the extra-special 2-group $2^{1+8}$. In particular, $\langle t_M, t_N\rangle$ acts trivially on the discriminant group 
$L^*/L  (\cong 2^8)$. Recall also from \cite{Gr2tod} that the quotient group $K.C_{O(L)}(t_M)/K$ of the centralizer of $t_M$ determines a maximal parabolic subgroup of the shape $2^6.\Omega^+_6(2)$ in  $\Omega^+_8(2)$, which corresponds to the stabilizer of a singular vector $u_M$. Similarly, $K.C_{O(L)}(t_N)/K\cong 2^6. \Omega^+_6(2)$ corresponds to the stabilizer of another singular vector $u_N$. 
Note that $g=t_Mt_N\in K$ and we have $N_{O(L)}(\langle g\rangle)/K \cong Sp_6(2)$, which is the stabilizer of a non-singular vector $u_M+u_N$.  Indeed, $N_{O(L)}(\langle g\rangle) \cong 2^{1+8}.Sp_6(2)$. By Theorem \ref{normalizer}, we also have the following lemma.

\begin{lemma}
Let $L=BW_{16}$ and $g\in O(L)$ a fourvolution. Then, 
\[
N_{\mathrm{Aut}(V_L)}(\langle \hat{g}\rangle) / \langle \hat{g}\rangle \cong 2^{8}. (2^7. Sp_6(2)). 
\]
\end{lemma}

\begin{remark}\label{Leech4a}
 We note that the Barnes-Wall lattice $BW_{16}$ can also be realized as the co-invariant lattice of the Leech lattice associated with a $4A$ element of the Conway group $Co_0$.  
 \end{remark}

Next we review the irreducible modules for $V_{BW_{16}}^{\hat{g}}$  for a fourvolution $g\in O(BW_{16})$.  Recall that 
there are two types of irreducible modules for $V_{BW_{16}}^{\hat{g}}$:  
\medskip

\noindent \textbf{Untwisted type:}  $V_{\lambda+L}(j),  \lambda\in L^*$  and $j=0,1,2,3$;  and 
\medskip

\noindent \textbf{Twisted type:} $V_{\lambda+L}^T[\hat{g}^i] (j), \lambda\in L^*, i=1,2,3$  and $j=0,1,2,3$. 
\medskip

Set $ F= 
\{ V_{L}^T[\hat{g}^i] (j) \mid i,j=0,1,2,3\} $ and $E= \{ V_{\lambda+L}(0)\mid  \lambda+L\in L^*/L\}$,  where $V_{L}^T[\hat{g}^0] (j)=V_{L}(j)$.

For $L=BW_{16}$, we have  $(1-g)L^* =L$ for any fourvolution $g$. 
By \cite{Lam19},  all irreducible modules of $V_{L}^{\hat{g}}$  are simple current modules. Moreover, the set of inequivalent irreducible modules $ \mathrm{Irr}(V_{BW_{16}}^{\hat{g}})$ of  $V_{BW_{16}}^{\hat{g}}$ forms an abelian group  under the fusion rules and it has a quadratic form $q$ defined by conformal weights modulo $\Z$. By the discusiion in \cite{Lam19}, we also know that  $E\cong 2^8$, $F\cong 4^2$ as abelian groups and $ \mathrm{Irr}(V_{BW_{16}}^{\hat{g}})= E\times F$. Note also that $F$ is the subgroup generated by 
$V_L(1)$ and $V_L^T[g](0)$.

\begin{theorem}
Let $L\cong BW_{16}$ and $g\in O(BW_{16})$ a fourvolution. Then 
$\Aut(V_L^{\hat{g}})\cong  [2^{15}].(Sp_6(2)\times 2)$.
\end{theorem}

\begin{proof}
By the same argument as in the proof of Theorem \ref{thm:4.4}, there is no automorphisms $s\in \aut(V_L^{\hat{g}})$ such that 
$V_L(1) \circ s \cong V_{\lambda+L}(j) $ for any $\lambda\in L^*\setminus L$ and $ 0\leq j\leq 3$.  In other words, if $s$ stabilizes the set of untwisted irreducible modules, then it also stabilizes the set $\{V_L(i)\mid i=0,1,2,3\}$ and is thus contained in 
\[
S= Stab_{\mathrm{Aut}(V_{BW_{16}}^{\hat{g}})} (\{V_L(i)\mid i=0,1,2,3\})\cong N_{\mathrm{Aut}(V_L)}(\langle \hat{g}\rangle) / \langle \hat{g}\rangle.
\]  
An extra automorphism $s$ thus maps $V_L(1)$ to a twisted type module.  
Note that $L^*/L \cong L/(1-g)L$ and  for any $\lambda+L\in L^*/L$, there is an $f\in Hom(L/(1-g)L, \C^*)$ such that 
$V_{\lambda+L}^T[\hat{g}](0) \cong V_L^T[\hat{g}](0) \circ f$. Without loss, we may assume $V_L(1) \circ s \cong V_L^T[\hat{g}](0)$ or $V_L^T[\hat{g}^3](0)$. In this case, $s$ stabilizes $F$. Note that  the orthogonal group  $O(F,q)\cong 2^2$. Moreover, there is an automorphism $\sigma\in \mathrm{Aut}(V_{BW_{16}}^{\hat{g}}) $ such that  $V_L(1) \circ \sigma \cong V_L^T[\hat{g}](0)$ by Remark \ref{Leech4a} and \cite[Theorem 3.4]{Lam2} and we have the desired result.  
  \end{proof}


\begin{thebibliography}{100}


\bibitem[ALY]{ALY}
T. Abe, C.H. Lam and H.Yamada, Extensions of tensor products of $\Z_p$-orbifold models of the lattice vertex operator algebra
$V_{\sqrt{2}A_{p-1}}$, \emph{J. Algebra} \textbf{510} (2018), 24--51.  


\bibitem[BLS]{BLS} K. Betsumiya, C.H. Lam and H. Shimakura, Automorphism group of orbifold vertex operator 
algebras associated with the Leech lattice : Non-prime cases, in preparation.

\bibitem[BLS2]{BLS2} K. Betsumiya, C.H. Lam and H. Shimakura, 
Holomorphic vertex 
operator 
algebras of central charge $24$ and their automorphism groups, in preparation.

\bibitem[BK04]{BK04}B. Bakalov, and V.\,G.\,Kac, Twisted modules over lattice vertex algebras, 
\emph{Lie theory and its applications in physics V}, 3--26, World Sci. Publ., River Edge, NJ, 2004.









\bibitem[CM]{CM} S. Carnahan and M. Miyamoto, 
Regularity of fixed-point vertex operator subalgebras; arXiv:1603.05645. 

\bibitem[CS99]{CS}
J.H.\ Conway and N.J.A.\ Sloane, Sphere packings, lattices and groups, 3rd Edition, Springer, New 
York, 1999.

\bibitem[Do93]{Do}
C.\ Dong, Vertex algebras associated with even lattices, {\it J. Algebra} {\bf 161} (1993), 245--265.







\bibitem[DL96]{DL} C.\ Dong and J.\ Lepowsky,
The algebraic structure of relative twisted vertex operators,
{\it J. Pure Appl. Algebra} {\bf 110} (1996), 259--295.

\bibitem[DLM96]{DLM}
C.\ Dong, H.\ Li,  and G.\ Mason, Simple Currents and Extensions of Vertex Operator
Algebras, {\it Comm. Math. Phys.} \textbf{180} (1996), 671-707.

\bibitem[DLM00]{DLM2}
C.\ Dong, H.\ Li, and G.\ Mason, Modular-invariance of trace functions
in orbifold theory and generalized Moonshine, 
\emph{Comm. Math. Phys.} {\bf 214} (2000), 1--56.










\bibitem[DN99]{DN}
C.\ Dong and K.\ Nagatomo, 
Automorphism groups and twisted modules for lattice vertex operator algebras, {\it in} Recent 
developments in quantum affine algebras and related topics (Raleigh, NC, 1998), 117--133,
{\it Contemp. Math.}, {\bf 248}, Amer. Math. Soc., Providence, RI, 1999. 


\bibitem[DRX17]{DRX}
C.\ Dong, L.\ Ren and F.\ Xu, On Orbifold Theory, \emph{Adv. Math.} {\bf 321} (2017), 1--30.




\bibitem[EMS18+]{EMS}
J. van Ekeren, S.\ M\"oller and N.\ Scheithauer, Construction and Classification of Holomorphic 
Vertex
Operator Algebras, \emph{J. Reine Angew. Math.}, Published Online. 


 

\bibitem[FLM88]{FLM}
I.\ Frenkel, J.\ Lepowsky and A.\ Meurman, Vertex operator algebras and the Monster, Pure and Appl.\ Math., Vol.134, Academic Press, Boston, 1988.




\bibitem[Gr98]{GO+102}  R.\,L.\,Griess, Jr.,   A vertex operator algebra related
to $E\sb 8$ with automorphism group ${\rm O}\sp +(10,2)$. {\it The Monster
and Lie algebras} (Columbus, OH, 1996),  43--58, Ohio State Univ. Math. Res.
Inst. Publ., 7, de Gruyter, Berlin, 1998.


\bibitem[Gr05]{Gr2tod} R.\,L.\,Griess, Jr.,Pieces of 2d: existence and uniqueness for Barnes-Wall and Ypsilanti lattices. Adv. Math. \textbf{196} (2005), No. 1, 147--192.

\bibitem[GL11]{GL}  R.\,L.\,Griess, Jr. and C.\,H.\,Lam, Dihedral groups and Sums of $EE_8$-lattices, \emph{Pure and applied mathematics quarterly} \textbf{7} (2011), No. 3, 621--743. 














\bibitem[Ka90]{Kac}
V.G.\ Kac, Infinite-dimensional Lie algebras, Third edition, Cambridge University Press, Cambridge, 
1990.







\bibitem[La20]{Lam19} C.H. Lam, Cyclic orbifold of lattice vertex operator algebras having group-like fusions, \emph{Lett. Math. Phys} \textbf{110 (5)} (2020), 1081-- 1112. 

\bibitem[La2]{Lam1} C.H. Lam, Automorphism group of an orbifold vertex operator 
algebra associated with the Leech lattice, in Proceedings of the Conference on Vertex Operator Algebras, Number Theory and Related Topics, \emph{Contemporary Mathematics}  \textbf{753} (2020),  127-- 138,  
American Mathematical Society, Providence, RI.


\bibitem[La3]{Lam2} C.H. Lam,  Some observations about the automorphism groups of certain orbifold vertex operator algebras, to appear in RIMS Kôkyûroku Bessatsu.

\bibitem[LS]{LS} C.H. Lam and H. Shimakura, Extra automorphisms of the cyclic orbifold of lattice vertex operator algebras, in preparation. 

\bibitem[LY14]{LY2} C.H. Lam and H. Yamauchi, On $3$-transposition groups generated by 
$\sigma$-involutions associated to $c=4/5$ Virasoro vectors, \emph{J. Algebra}, 416 (2014), 84-121. 




%
\bibitem[Le85]{Le} J.\ Lepowsky, Calculus of twisted vertex operators,  {\it Proc.
Natl. Acad. Sci. USA} {\bf 82} (1985), 8295--8299.




\bibitem[Mi15]{Mi}
M.\ Miyamoto, $C_2$-cofiniteness of cyclic-orbifold models, {\it Comm. Math. Phys.} {\bf 335} (2015),  1279--1286.




\bibitem[Sh04]{Sh04}
H.~Shimakura, The automorphism group of the vertex operator algebra $V_L^+$ for an even lattice $L$ without roots, {\it J. Algebra} {\bf 280} (2004), 29--57.

\bibitem[Sh06]{Sh06}
H.~Shimakura,  The automorphism groups of the vertex operator algebras V+L: general case, \emph{Math. Z. }        \textbf{252} (2006), no. 4, 849 -- 862. 

\bibitem[Sh07]{Sh07}
H.~Shimakura, Lifts of automorphisms of vertex operator algebras in
  simple current extensions, \emph{Math. Z.} \textbf{256} (2007), no.~3, 491--508.




\end{thebibliography}
\end{document}